\documentclass[10pt]{article}

\usepackage{amsmath}    

\usepackage[colorlinks=true,linkcolor=red,citecolor=blue]{hyperref}

\usepackage{amsthm}     
\usepackage{amssymb}    
\usepackage{braket}     
\usepackage{mathtools}
\usepackage{relsize}

\usepackage[all]{xy}

\newcommand{\what}{\widehat}

\newcommand{\abs}[1]{\left| #1 \right|}
\newcommand{\st}{\circledast}

\newcommand{\C}{\mathbb C}

\newcommand{\op}{\operatorname}
\newcommand{\mbf}{\mathbf}

\newcommand{\mc}{\mathcal}
\newcommand{\mf}{\mathfrak}
\newcommand{\bds}{\boldsymbol}

\renewcommand{\2}{\mathsmaller{(2)}}

\renewcommand{\ker}{\op{Ker}}
\newcommand{\im}{\op{Im}}
\newcommand{\tr}{\op{tr}}
\newcommand{\botimes}{\,\bar\otimes}

\DeclareMathOperator{\Aut}{Aut} 
\DeclareMathOperator{\End}{End}

\newtheorem{thm}{Theorem}[section]
\newtheorem{cor}[thm]{Corollary}
\newtheorem{lem}[thm]{Lemma}
\newtheorem{prop}[thm]{Proposition}

\theoremstyle{definition}
\newtheorem{defn}[thm]{Definition}
\newtheorem{ex}[thm]{Example}
\newtheorem{rem}[thm]{Remark}

\numberwithin{equation}{section}

\title{$L^{2}$-homology for inclusions of von Neumann algebras}
\author{Miguel Berm\'udez\thanks{\texttt{bermudez@math.jussieu.fr}}\medskip\\ 
{\small Institut de Math\'ematiques de Jussieu}\\ {\small Universit\'e Paris-Diderot, Paris} }
                                           
\begin{document}
\maketitle

\begin{abstract}
 In this paper we define $L^{2}$-homology and $L^{2}$-Betti numbers for a tracial $*$-algebra $A$ with respect to a von Neumann subalgebra $B$. When $B$ is reduced to the field of complex numbers we recover the $L^{2}$-Betti numbers of $A$ as defined by A. Connes and D. Shlyakhtenko in \cite{ConSh}, but we will show that taking into account the von Neumann subalgebra $B$ yields to a number of advantages like, for instance, a much better behavior with respect to compression and directed sums. 
  
 Our main result is that $L^{2}$-homology and $L^{2}$-Betti numbers of discrete measured groupoids and equivalence relations as defined by D. Gaboriau \cite{GabL2} and R. Sauer \cite{Sau} coincide with those of their convolution algebras. 
 
 We also define new invariants for inclusions of von Neumann algebras, which we call {\em residual $L^{2}$-Betti numbers}. We prove that the residual $L^{2}$-Betti numbers of a finite factor with respect to Cartan subalgebra coincide with the $L^{2}$-Betti numbers of the standard equivalence relation associated to the inclusion.
 \end{abstract}

\section{Introduction}
$L^{2}$-Betti numbers for von Neumann algebras and, more generally, for tracial $*$-algebras, has been introduced by A. Connes and D. Shlyakhtenko in \cite{ConSh} (see also the previous work of W. Paschke \cite{MR1407708}) and a quite important amount of work has been done on their study in recent years. Unfortunately, their main objective, that was distinguishing von Neumann algebras of free groups, is far from being reached. Actually, a continuous version of such invariants, introduced by A. Thom in \cite{MR2399103} has been shown to vanish by Thom himself for von Neumann algebras with a Cartan subalgebra, by V. Alekseev and D. Kyed \cite{Alekseev:2011aa} for property (T) von Neumann algebras  and more recently by V. Alekseev \cite{Alekseev:2013aa} for free group factors and by S. Popa and S. Vaes \cite{Popa:2014aa} for any $\rm II_{1}$-factor. In the purely algebraic setting, \cite{ConSh,MR2399103}, the only nontrivial result is that $L^{2}$-Betti numbers vanish for commutative \cite{ConSh} and more generally diffuse center von Neumann algebras \cite{MR2399103}.

In this paper we propose a generalization of such invariants for inclusions of tracial von Neumann algebras, what we call here {\em tracial extensions}. The idea, as in \cite{ConSh}, is to consider the relative Hochschild complex of the extension with values in some suitable bimodule, such that the corresponding homology groups are non trivial and had a module structure over some natural tracial von Neumann algebra with respect to which one can take a dimension. The first difficulty is then that, except in the central case (i.e. when $A$ is a $B$-algebra), there is no obvious such a module. One could be tempted to take the algebra $A$ itself; this results on the usual relative Hochschild homology groups, which are highly non trivial in many interesting examples. But the only von Neumann algebra acting on them in a natural way is the center of $A$, and one can prove that the corresponding $L^{2}$-Betti numbers are most of the time zero or infinite.

Let $A$ be a tracial $*$-algebra endowed with a von Neumann (i.e. weakly closed) sub-algebra $B\subset A$. We associate to every such a pair $A/B$ a new tracial $*$-algebra, noted $A*_{B}A$ and called the {\em fiber square of $A/B$}, formed by $A$-bimodular endomorphisms of $A\otimes_{B}A$. If $A\st_{B}A=W^{*}(A*_{B}A)$ denotes the corresponding von Neumann algebra, then the vector space 
\begin{equation}\label{}
A\hat\otimes_{B}A=(A\otimes_{B} A)\otimes_{A*_{B}A}(A\st_{B}A)
\end{equation}
has an obvious structure of $A\otimes A^{o}$-$A\st_{B}A$-bimodule; by functoriality, the relative Hochschild homology groups of $A/B$ with values in $A\hat\otimes_{B}A$, which will be called the {\em $L^{2}$-homology groups} of $A/B$ and noted $H_{\bullet}^{\2}(A/B)$, inherit a natural structure of $A\st_{B}A$-modules. The $L^{2}$-Betti numbers of $A/B$ are then defined by
\begin{equation*}
\beta^{(2)}_{\bullet}(A/B)=\dim_{A\st_{B}A}H^{(2)}_{\bullet}\left(A/B\right).
\end{equation*}
At first sight one might think that the natural choice of bimodule is the Connes-Sauvageot tensor product $L^{2}A\botimes_{B}L^{2}A$, which is also endowed with a natural $A$-bimodular action of $A\st_{B}A$. Unfortunately, a careful analysis shows that the corresponding Hochschild homology has infinite von Neumann dimension in most of the interesting examples. Hence $A\hat\otimes_{B}A$ appears to be the only natural $A$-bimodule yielding to non trivial relative $L^{2}$-invariants.

We shall see that the $L^{2}$-Betti numbers of the trivial extension $A/\C$ coincide with those of the tracial algebra $A$ as defined by Connes and Shlyakhtenko in \cite{ConSh}, but our point of view has a number of advantages even in this case. For instance, relative $L^{2}$-Betti numbers turn out to behave much better than Connes-Shlyakhtenko ones, since they transform linearly with respect to directed sums and compressions. Actually, the right definition of $L^{2}$-invariants of a tracial algebra $A$ should be that of the corresponding maximal central extension $A/Z(A)$. Of course, if $A$ is a factor, both definitions coincide.

One interesting feature of relative $L^{2}$-invariants of tracial extensions is that they give rise to a purely algebraic description of $L^{2}$-invariants of countable measure preserving standard equivalence relations and groupoids in the sense of Gaboriau \cite{GabL2} and Sauer \cite{Sau}. In order to precise this, let $\mc N_{A/B}$ be the normalizing algebra of the extension $A/B$, i.e. the vector span of the group $\mbf U_{A/B}$ of unitaries $u\in A$ such that $u^{*}Bu=B$. The {\em residual $L^{2}$-Betti numbers of $A/B$} are defined as the $L^{2}$-Betti numbers of $\mc N_{A/B}/B$, namely 
$$
\nabla^{(2)}_{\bullet}(A/B)=\beta^{(2)}_{\bullet}(\mc N_{A/B}/B).
$$ 
Now let us assume that $A$ is a finite factor and $B$ a Cartan subalgebra, i.e. a maximal abelian subalgebra with dense normalizing algebra. Since $B$ is abelian, there exists a probability standard space $X$ such that $B\simeq L^{\infty}X$ and the corresponding measure preserving action of $\mbf U_{A/B}$ on $X$ turns out to be countably generated, which implies that the orbit equivalence relation $\mc R_{A/B}$ is discrete.

\begin{thm}\label{thm:residual-L2-Betti-numbers}
If $A$ is a finite factor endowed with a Cartan subalgebra $B$, then the residual $L^{2}$-Betti numbers of $A/B$ coincide with the $L^{2}$-Betti numbers of $\mc R_{A/B}$ in the sense of Gaboriau \cite{GabL2}. Namely 
 \begin{equation*}
\nabla^{(2)}_{\bullet}\left(A/B \right)=\beta_{\bullet}^{(2)}(\mc R_{A/B})
\end{equation*}
\end{thm}

We will actually prove that the Hochschild $L^{2}$-complex of the tracial extension and the geometric $L^2$-complex of the equivalence relation $\mc R_{A/B}$ turn out to be isomorphic. This is extremely surprising, since they belong to two apparently very different categories of chain complexes. Indeed, the former is a chain complex of $A\st_{B}A$-modules, and the latter a chain complex of modules over the von Neumann algebra of $\mc R_{A/B}$. We shall prove that, in this particular case, these algebras are incidentally isomorphic.  

More generally, to every standard groupoid $\mc G$ over a standard probability space $X$ (also called a discrete measured groupoid in \cite{Sau}) we can naturally associate a tracial extension $\C\mc G/L^{\infty}X$ such that $L^{\infty}X$ is endowed with a structure of $\C\mc G$-module. This leads R. Sauer to define the $L^{2}$-Betti numbers of $\mc G$ as
$$
\beta^{(2)}_{\bullet}(\mc G):=\dim_{N\mc G}H^{(2)}_{\bullet}(\mc G)
$$
where $H^{(2)}_{\bullet}(\mc G):=\op{Tor}_{\bullet}^{\C\mc G}(L^{\infty}X,N\mc G)$ and $N\mc G=W^{*}(\C\mc G)$ is the enveloping von Neumann algebra of $\C\mc G$. For instance, every countable group can be viewed as a discrete measured groupoid over a single point and the numbers above coincide with the $L^{2}$-Betti numbers of countable groups if the sense of J. Cheeger and M. Gromov \cite{MR837621}. Further, Neshveyev and Rustad proved in \cite{NeshRus} that, if $\mc G$ is an equivalence relation, then both Sauer's and Gaboriau's definitions coincide. In \cite{MR2650795} R. Sauer and A. Thom study $L^{2}$-invariants for standard equivalence relations and groupoids using homological algebra approaches and they get some interesting applications to group theory and algebraic topology. In this paper we shall prove the following:
\begin{thm}\label{thm:L2=L2-groupoids}
If $\mc G$ is a discrete measured groupoid over a standard probability space $X$, then
\begin{equation*}
\beta^{(2)}_{\bullet}(\mc G)=\beta^{(2)}_{\bullet}(\C\mc G/L^{\infty}X)
\end{equation*}
\end{thm}
Unlike the previous case, the geometric $L^{2}$-complex of $\mc G$ is no longer isomorphic to its Hochschild partner. Actually the von Neumann algebras $N\mc G$ and $N\mc G\st_{\bds L^{\infty}X}N\mc G$ shall be shown to be isomorphic if and only if $\mc G$ is a standard equivalence relation.

\section{Fiber products of tracial extensions}

Let $B$ be a von Neumann algebra. A \emph{hermitian right $B$-module} is a right $B$-module together with a compatible $B$-valued inner product, i.e. a sesquilinear form $(m,n)\in M\times M\mapsto m^{*}n\in B$ such that:
\begin{enumerate}
 \item $m^{*}(nx)=(m^{*}n)x$ for every $m,n\in M$ and $x\in B$;
 \item $(m^{*}n)^{*}=n^{*}m$ for every $m,n\in M$;
 \item $m^{*}m\geq 0$ for every $m\in M$;
 \item  $m^{*}m=0$ if and only if $m=0$.
\end{enumerate}
A hermitian left $B$-module is defined as a hermitian right module over the opposite algebra $B^{o}$. In that case we note $xm=mx^{o}$ and $m^{*}n=nm^{*}$. A hermitian $B$-bimodule is an algebraic $B$-bimodule together with a left and a right compatible $B$-valued inner products. The algebraic tensor product $M\otimes_{B} N$ of two hermitian $B$-bimodules $M$ and $N$ can be endowed with a structure of hermitian $B$-bimodule by setting\footnote{One can easily see that these sesquilinear forms are well defined on the algebraic tensor product $M\otimes_{B}N$, but it is not obvious that they are positive and non degenerate. A proof of this fact can be found in \cite[Proposition 4.5]{Lance}.} $(m\otimes n)^{*}(m'\otimes n')=(m^{*}m')(n'n^{*})$ and $(m'\otimes n')(m\otimes n)^{*}=(n'n^{*})(m^{*}m')$. If $M$ is a hermitian right module over a tracial von Neumann algebra $B$, then the formula $\braket{m|n}_{B}:=\tr_{B}(m^{*}n)$ defines a complex hermitian product on $M$. The action of $B$ in $M$ is bounded with respect to the corresponding norm and extends to the Hilbert completion $L^{2}M$, which can thus be viewed as a left $B$-module. But, in general, there is no compatible $B$-valued inner product on $L^{2}M$. If $M$ is a hermitian $B$-bimodule, then the trace on $B$ induce two eventually different complex hermitian products on $M$. If they coincide, then we say that $M$ is a \emph{trace-symmetric bimodule}. Observe that the tensor product of two $B$-bimodules $M$ and $N$ is an example of trace-symmetric bimodule. The Hilbert space $L^{2}(M\otimes_{B}N)$ coincides with the Connes-Sauvageot tensor product of the Hilbert bimodules $L^{2}M$ and $L^{2}N$ (cf. \cite[Appendix B.$\delta$]{MR1303779} and \cite{MR703809}). For further details on the theory of hermitian $C^{*}$-modules, the reader may consult \cite{Lance}.

A \emph{tracial $*$-algebra} is a unital $*$-algebra $A$ together with a normalized trace $\op{tr}_{A}$ verifying the following properties:
\begin{enumerate}
 \item Faithfulness : the sesquilinear form $\braket{a|b}=\op{tr}_{A}(a^{*}b)$ is a hermitian complex product on $A$;
 \item Boundedness : the left and right actions of $A$ over itself are bounded with respect to the pre-Hilbert norm, such that they extend into a structure of $A$-bimodule on the Hilbert completion $L^{2}A$, called the GNS representation of $A$;
\end{enumerate}
If the inclusion of $A$ in $\mf B(L^{2}A)$ induced by the GNS representation is weakly closed then $A$ is called a \emph{tracial von Neumann algebra}. To every tracial $*$-algebra $A$ corresponds a canonical von Neumann tracial algebra $W^{*}(A)$ given by the weak closure of $A$ in $\mf B(L^{2}A)$ endowed with the trace $\op{tr}_{W^{*}(A)}(a)=\braket{a(\mbf 1)|\mbf 1}$. Moreover one has $L^{2}W^{*}(A)=L^{2}A$. 

\begin{defn}
We shall call \emph{tracial extension} the pair $A/B$ formed by a tracial $*$-algebra $A$ together with a von Neumann subalgebra $B\subset A$. 
\end{defn}

In that case $L^{2}B$ is a closed subspace of $L^{2}A$, and one can prove that the orthogonal projection $L^{2}A\to L^{2}B$ induces a trace preserving conditional expectation $\mbf E:A\to B$, i.e. a completely positive $B$-bimodular map such that $\op{tr}_{A}=\op{tr}_{B}\circ \mbf E$. Let $\mbf U_{A/B}$ be the {\em normalizer} of $B$ in $A$ i.e. the set of unitaries $u\in A$ such that $u^{*}Bu=B$. Its vector span $\mc N_{A/B}$ is an involutive subalgebra of $A$ called the {\em normalizing algebra} of the extension. The algebra $\mc N_{A/\C}$ will be called the {\em residual algebra} of $A$. Since $B$ is a von Neumann algebra, it coincides with its residual algebra and is thus contained in $\mc N_{A/B}$. Since the conditional expectation is defined as the restriction to $A$ of the orthogonal projection on $L^{2}A\to L^{2}B$, then for every $a\in A$ and $u\in \mbf U_{A/B}$ we have $\mbf E(uau^{*})=u^{*}\mbf E(a)u$. 
Moreover $A$ can be regarded as a hermitian $B$-bimodule with respect to the inner products $(a,b)\mapsto \mbf E(a^{*}b)$ and $(a,b)\mapsto \mbf E(ba^{*})$. 

If we consider another tracial extension $C/B$, then the complex hermitian product on the trace-symmetric $B$-bimodule $A\otimes_{B}C$ is explicitly given by the following formulas:
\begin{align}\label{eq:inner-tensor}
\braket{a\otimes b|c\otimes d}
&= \op{tr}_{B}\left(\mbf E(db^{*})\,\mbf E(a^{*}c)\right) \\
&= \op{tr}_{C}\left(b^{*}\,\mbf E(a^{*}c)\,d\right) \nonumber \\
&= \op{tr}_{A}\left(c\,\mbf E(db^{*})\,a^{*}\right). \nonumber
\end{align}
For every element in the set 
\begin{equation}\label{eq:S-A/B}
\mbf S_{B}^{A,C}=\set{(u,v)\in \mbf U_{A/B}\times \mbf U_{C/B}|u^{*}xu=vxv^{*},\forall x\in B}
\end{equation}
we have a $A$-$C$-bimodular endomorphism of $A\otimes_{B}C$ defined on elementary tensors by the formula $(u *v) (a\otimes b)=a u\otimes vb$. We shall write $\mbf S_{A/B}$ instead of $\mbf S^{A,A}_{B}$. 

\begin{defn}
The \emph{algebraic fiber product} of two tracial extensions $A/B$ and $C/B$ is the algebra generated by the operators $u*v$, namely
\begin{equation}
A*_{B}C:= \Braket{u*v|(u,v)\in \mbf S_{B}^{A,C}}\subset \End_{A-C}\left(A\otimes_{B}C\right )
\end{equation}
which can naturally be viewed as a $*$-algebra since the involution $(u*v)^{*}=u^{*}*v^{*}$ extends to $A*_{B}C$. 
\end{defn}

\begin{prop}\label{prop:algebraic-fiber-product}
Let $A$ and $C$ be two tracial extensions of the same von Neumann algebra $B$. Then:
\begin{enumerate}
\item The vector $\mbf 1\otimes\mbf 1\in A\otimes_{B}C$ is separating and tracial for $A*_{B}C$;

\item The space of $B$-invariant vectors of $A\otimes_{B}C$ is a faithful $A*_{B}C$-submodule;

\item The algebra $A *_{B}C$ contains the center of $B$ through its action on $A\otimes_{B}C$ given by $\what x(a\otimes b)=ax\otimes b=a\otimes xb$;
  
\item If $B$ is commutative and the extensions $A/B$ and $C/B$ are central, then $A*_{B}C$ is isomorphic to the tensor product $\mc A^{o}\otimes_{B}\mc C$ where $\mc A$ and $\mc C$ denote the residual algebras of $A$ and $C$. 
\end{enumerate}
\end{prop}
\begin{proof}

(i) The vector $\mbf 1\otimes\mbf 1$ is obviously cyclic for the von Neumann algebra generated by  $A\otimes C^{o}$ on $L^{2}(A\otimes_{B} C)$, hence separating for its commutant, which contains by construction $A*_{B}C$. Let $\phi$ be the faithful state on $A*_{B}C$ associated to $\mbf 1\otimes\mbf 1$. It remains to prove that it is tracial, i.e. that $\phi(ST)=\phi(TS)$ for every $T,S\in A *_{B}C$. By linearity, one can assume that $S=u*v$ and $T=w*t$ for some $(u,v),(w,t)\in \mbf S_{B}^{A,C}$. In that case one has
\begin{align*}
\phi(TS) 
&=  \op{tr}_{A}(w^{*}u^{*}\mbf E(vt))
= \op{tr}_{A}(u^{*}\mbf E(vt)w^{*})\\
&= \op{tr}_{A}(u^{*}w^{*}w\mbf E(vt)w^{*}) 
= \op{tr}_{A}(u^{*}w^{*}t\mbf E(vt)t^{*}) \\
&= \op{tr}_{A}(u^{*}\mbf E(vt)w^{*}ww^{*}) 
= \op{tr}_{A}(u^{*}w^{*}w\mbf E(vt)w^{*}) & (w^{*}w\in Z(B))\\
&= \op{tr}_{A}(u^{*}w^{*}t\mbf E(vt)t^{*}) 
= \op{tr}_{A}(u^{*}w^{*}\mbf E(tvtt^{*})) \\
& = \op{tr}_{A}(u^{*}w^{*}\mbf E(tv)tt^{*})
= \op{tr}_{A}(u^{*}w^{*}tt^{*}\mbf E(tv)) & (tt^{*}\in Z(B))\\
&= \op{tr}_{A}(u^{*}w^{*}ww^{*}\mbf E(tv)) = \op{tr}_{A}(u^{*}w^{*}\mbf E(tv))=\phi(ST)
\end{align*}
which proves the assertion.

(ii) Let $(A\otimes_{B}C)^{B}$ be the set of $B$-invariant vectors of $A\otimes_{B}C$. For every $x\in B$, $T\in A*_{B}C$ and $\xi\in (A\otimes_{B}C)^{B}$ we have $xT(\xi)=T(x\xi)=T(\xi x)=T(\xi)x$. This implies that $(A\otimes_{B}C)^{B}$ is a $A *_{B}C$-invariant space. Furthermore, it is faithful since it obviously contains the separating vector $\mbf 1\otimes\mbf 1$.

(iii) By definition $\mbf U_{Z(B)}\subset \mbf U_{A/B}\cap \mbf U_{C/B}$. Furthermore, for every $x\in \mbf U_{Z(B)}$ one has $(x,\mbf 1),(\mbf 1,x)\in \mbf S_{B}^{A,C}$ and  $\what x=\mbf 1*x=x*\mbf 1$. Since $Z(B)$ is the vector span of its unitaries, the assertion follows.

(iv) Let us consider the linear map  
\begin{equation}\label{}
\phi:A*_{B} C\to A\otimes_{B}C\quad,\quad T\mapsto T(\mbf 1\otimes \mbf 1)
\end{equation}
which is injective by (i). In this case $\mbf S_{B}^{A,C}=\mbf U_{A}\times \mbf U_{C}$ and every $T\in A*_{B}C$ can therefore be written as a linear combination of operators $u*v$ with $u\in \mbf U_{A}$ and $v\in \mbf U_{C}$. Since $(u*v)(u'*v')=(u'u)*(vv')$ and $u*v(\mbf 1\otimes\mbf 1)=u\otimes v$, it follows that $\phi$ maps isomorphically $A*_{B}C$ onto $\mc A^{o}\otimes_{B}\mc C$.
\end{proof}

\subsection{Directed sums}
A {\em weighted family of tracial algebras} is a finite sequence $(A_{n},\alpha_{n})_{n}$, such that $A_{n}$ is a tracial $*$-algebra for every $n$ and $(\alpha_{n})_{n}$ is a partition of unity, i.e. $\alpha_{n}\in [0,1]$ for each $n$ and $\sum_{n}\alpha_{n}=1$. The directed sum algebra $\bigoplus_{n}A_{n}$, endowed with the trace $\tr_{A}=\sum_{n}\alpha_{n}\tr_{A_{n}}$, is a tracial algebra that will be noted $\sum_{n}\alpha_{n} A_{n}$ and called the {\em weighted directed sum} of the $A_{n}$'s.

\begin{prop}\label{prop:fiber-product-directed-sums}
Let $A_{n}/B_{n}$, $C_{n}/B_{n}$ be two finite sets of tracial extensions, and let $\alpha_{n}$ be a partition of unity. Then there exists a natural isomorphism of tracial algebras
\begin{equation}\label{}
\left(\sum_{n}\alpha_{n} A_{n}\right)*_{\sum_{n} \alpha_{n} B_{n}}\left(\sum_{n}\alpha_{n} C_{n}\right)\simeq \sum_{n}\alpha_{n}(A_{n}*_{B_{n}}C_{n})
\end{equation}
\end{prop}
\begin{proof}
Let $A=\sum_{n}\alpha_{n}A_{n}$, $B=\sum_{n}\alpha_{n}B_{n}$ and $C=\sum_{n}\alpha_{n}C_{n}$. We have an obvious identification of $A$-$C$-bimodules between $A\otimes_{B}C$ and $\bigoplus_{n}(A_{n}\otimes_{B_{n}}C_{n})$. Moreover, every unitary element of $u\in \oplus_{n}A_{n}$ is a sum of unitaries $u_{n}\in A_{n}$, and $uxu^{*}=\sum_{n}u_{n}x_{n}u_{n}^{*}$ for every $x=\sum_{n}x_{n}\in B$. Hence that $u\in \mbf U_{A/B}$ if and only if $u_{n}\in \mbf U_{A_{n}/B_{n}}$ for each $n$. A similar argument shows that $(u,v)\in \mbf S_{B}^{A,C}$ if and only if $(u_{n},v_{n})\in \mbf S_{B_{n}}^{A_{n},C_{n}}$ for each $n$. It follows that $A*_{B}C=\bigoplus_{n}(A_{n}*_{B_{n}}C_{n})$. Furthermore, the conditional expectation of the tracial extension $A/B$ turns out to be the directed sum of the conditional expectations of $A_{n}/B_{n}$. Hence, for every $u=\sum_{n}u_{n}\in \mbf U_{A/B}$ and $v=\sum_{n}v_{n}\in \mbf U_{C/B}$ one gets 
\begin{multline}\label{}
\tr_{A*_{B}C}(u*v)=\tr_{A}(u^{*}\mbf E(v))\\= \sum_{n}\alpha_{n}\tr_{A_{n}}(u_{n}^{*}\mbf E(v_{n}))=\sum_{n}\alpha_{n}\tr_{A_{n}*_{B_{n}}C_{n}}(u_{n}*v_{n})
\end{multline}
which finishes the proof.
\end{proof}

Let us assume now that all the algebras $B_{n}$ are isomorphic to some commutative von Neumann algebra $B$. Then $A=\sum_{n}\alpha_{n}A_{n}$ can be viewed as a tracial extension of $B$ in a natural way, but in that case the conditional expectation of $A/B$ is no longer the directed sum but the weighted sum of the conditional expectations of $A_{n}/B_{n}$, namely
\begin{equation}\label{}
\mbf E\left(\sum_{n}a_{n}\right)=\sum_{n}\alpha_{n}\mbf E(a_{n}).
\end{equation}

\begin{prop}\label{prop:fiber-product-central-directed-sums}
 Let $A_{n}/B$ and $C_{n}/B$ be a finite sequence of central tracial extensions of the same commutative von Neumann algebra $B$. Then for every pair of partitions of unity $(\alpha_{n})$ and $(\beta_{n})$ we have an isomorphism of tracial algebras
 \begin{equation}\label{}
\left(\sum_{n}\alpha_{n}A_{n}\right)*_{B}\left(\sum_{n}\beta_{n}C_{n}\right)\simeq\sum_{i,j}\alpha_{i}\beta_{j}(A_{i}*_{B}C_{j})
\end{equation}
\end{prop}
\begin{proof}
Let us note $\mc A$, $\mc C$, $\mc A_{i}$ and $\mc C_{i}$ the residual algebras of $A$, $C$, $A_{i}$ and $C_{i}$ respectively. By item (iv) in proposition \ref{prop:algebraic-fiber-product}, we have $A*_{B}C=\mc A^{o}\otimes_{B}\mc C$ and $A_{i}*_{B}C_{j}=\mc A_{i}^{o}\otimes_{B}\mc C_{j}$. Since $\mc A=\oplus_{n}\mc A_{n}$ and $\mc C=\oplus_{n} \mc C_{n}$, it follows that $A*_{B}C\simeq \oplus_{i,j}A_{i}*_{B}C_{i}$. For every $a=\sum_{i}a_{i}\in \mc A$ and $c=\sum_{n}c_{n}\in \mc C$ we have 
\begin{align}\label{}
\tr_{\mc A^{o}\otimes_{B}\mc C}(a^{o} \otimes c)
	&=\tr_{B}(\mbf E(a)\mbf E(c))\\
	&=\sum_{ij}\alpha_{i}\beta_{j}\tr_{B}(\mbf E(a_{i})\mbf E(c_{j}))\\
	&=\sum_{ij}\alpha_{i}\beta_{j}\tr_{\mc A_{i}^{o}\otimes_{B}\mc C_{j}}(a_{i}^{o}\otimes c_{j})
\end{align}
which finishes the proof.
\end{proof}

\section{Relative $L^{2}$-invariants}

Let $A$ be a tracial von Neumann algebra. A \emph{Hilbert $A$-module} is a Hilbert space $H$ together with a weakly continuous $*$-representation $A\to \mf B(H)$. It is said to be \emph{finitely generated} if $H$ can be isometrically embedded as a submodule of $L^{2}A\otimes\C^{n}$ for some $n$. The commutant of $A$ in $L^{2}A\otimes\C^{n}$ is naturally isomorphic to the finite von Neumann algebra $\mbf M_{n}(A^{o})$ of square matrices with entries in the opposite algebra $A^{o}$. The (non normalized) trace of a matrix $a=(a_{ij})\in \mbf M_{n}(A^{o})$ is defined as $\op{Tr}(a)=\sum_{i=1}^{n}\op{tr}_{A}(a_{ii}^{o})=\sum_{i}^{n}\tr_{A^{o}}(a_{ii})$. For every isometric $A$-modular embedding $\theta:H\to L^{2}A\otimes \C^{n}$, the orthogonal projection onto $H$, i.e. the operator $\theta^{*}\theta$,  belongs to the commutant of $A$ in $L^{2}A\otimes\C^{n}$; the von Neumann dimension of $H$ is then defined as $\dim_{A}H=\op{Tr} \theta^{*}\theta$. Since two such projections are unitarily equivalent, they have the same trace; in particular the von Neumann dimension of a finitely generated Hilbert $A$-module does not depend on the embedding in $L^{2}A\otimes \C^{n}$. More generally, one can define the von Neumann dimension of any algebraic $A$-module. The idea, due to Luck (cf. \cite{LuckBook}), is the following: if $P$ is a finitely generated projective $A$-module, then $L^{2}A\otimes_{A}P$ is a finitely generated Hilbert $A$-module. The \emph{generalized von Neumann dimension} of an arbitrary $A$-module $M$ shall be defined by
\begin{equation}\label{}
\dim_{A}M:=\textstyle\sup_{P}\dim_{A}L^{2}A\otimes_{A}P
\end{equation}
where $P$ runs over all finitely generated projective submodules of $M$. A morphism of $A$-modules $\phi:M\to N$ shall be called a {\em dimension isomorphism} if the von Neumann dimension of its kernel and cokernel vanishes.

\bigskip
A \emph{presimplicial set} is a sequence of sets $\mc E_{\bullet}=\set{\mc E_{n}\vert n\geq 0}$ endowed with \emph{face maps} $\pi_{i}:\mc E_{n}\to \mc E_{n-1}$ for $n\geq 1$ and $i=0,\dots,n$, satisfying the following relations $\pi_{i}\pi_{j}=\pi_{j-1}\pi_{i}\quad\text{for}\quad i<j$ \footnote{A simplicial set can also be defined as a contravariant functor $\mc E$ from the small category of order preserving maps between non-empty finite ordinals (called the simplex category) to the category of sets, where $\mc E_{n}=\mc F(n+1)$ and $\pi_{i}:\mc E_{n}\to \mc E_{n-1}$ are the maps induced by the different order preserving inclusions from $n$ to $n+1$.}. If the sets $\mc E_{n}$ and their face maps belong to the category of abelian groups, the sequence can be turned into a positive chain complex $(\mc E_{\bullet},d_{\bullet})$ with boundary $d_{n}:\mc E_{n}\to \mc E_{n-1}$ given by
$d_{n}=\sum_{i=0}^{n}(-1)^{i}\pi_{i}$. The corresponding sequence of homology groups shall be noted by $H_{\bullet}(\mc E)$.

\medskip
Let $A/B$ be an extension of unital algebras and let us consider the following presimplicial spaces:
\begin{enumerate}
\item The \emph{bar complex} of  $A/B$ is the presimplicial $A$-bimodule given by the sequence $\mbf K_{n}(A/B)=A^{\otimes_{B}(n+2)}$ together with the face maps 
\begin{equation}\label{eq:bar-face-maps}
\mbf k_{n,i}(a_{0}\otimes\cdots\otimes a_{n+1}) =a_{0}\otimes\cdots\otimes a_{i}a_{i+1}\otimes\cdots\otimes a_{n+1}
\end{equation}
The corresponding chain complex $(\mbf K_{\bullet}(A/B),d)$ is called the \emph{bar complex} of $A/B$. If $\mu:A\otimes_{B}A\to A$ denotes the multiplication on $A$, then the corresponding augmented complex
\begin{equation}\label{bar-resolution}
0\to\mbf K_{\bullet}(A/B)\xrightarrow{d}\mbf K_{\bullet}(A/B)\xrightarrow{\mu} A\to 0
\end{equation}
is exact, i.e. it defines a resolution of $A$ in the category of $A$-bimodules. Given a $A$-bimodule $M$, the \emph{Hochschild $M$-valued complex of $A/B$} is the chain complex $C_{\bullet}(A/B:M)= M\otimes_{A\otimes A^{o}} \mbf K_{\bullet}(A/B)$. The corresponding homology is called the \emph{Hochschild homology of $A/B$ with values in $M$} and noted $H_{\bullet}(A/B:M)$. The $n$-chain space $C_{n}(A/B:M)$ is naturally isomorphic to the space of coinvariants of the $B$-bimodule $M\otimes_{B} A\otimes_{B}\overset{n}\dots\otimes_{B}A$; the Hochschild face maps are given by 
\begin{equation}\label{eq:Hochschild-face-maps}
\mbf h_{n,i}(a_{0}\otimes\dots\otimes a_{n})=
\begin{cases}
 a_{0}\otimes\cdots\otimes a_{i}a_{i+1}\otimes \dots\otimes a_{n}& i=0,\dots,n-1\\
 a_{n}a_{0}\otimes\cdots\otimes a_{n-1} & i=n
\end{cases}
\end{equation}
where we take $a_{0}\in M$ and $a_{1},\dots,a_{n}\in A$. In the particular case $M=A$ we simply write $C_{\bullet}(A/B)$ instead of $C_{\bullet}(A/B:A)$ and $HH_{\bullet}(A/B)$ instead of $H_{\bullet}(A/B:A)$.

\item  The {\em acyclic Hochschild complex} of $A/B$ is defined as the presimplicial $A*_{B}A$-module $\mbf Z_{\bullet}(A/B)=C_{\bullet}(A/B:A\otimes_{B}A)$. More precisely, we have $\mbf Z_{n}(A/B)=C_{n+1}(A/B)$ together with the face maps $\mbf z_{n,i}=\mbf h_{n+1,i+1}$. 
\end{enumerate}

\begin{prop}\label{prop:bar-classifying-acyclic}
 The bar and acyclic Hochschild complexes of any extension of unital algebras $A/B$ are acyclic and define resolutions of $A$ and $A_{B}=A/[B,A]$ respectively.
\end{prop}
\begin{proof}
One can easily check that $H_{0}(\mbf K(A/B))=A$ and $H_{0}(\mbf Z(A/B))=A_{B}$. If we set $\mbf K_{-1}(A/B)=A$ and $\mbf Z_{-1}(A/B)=A_{B}$, then we obtain contracting chain homotopies $r_{n}:\mbf K_{n}(A/B)\to \mbf K_{n+1}(A/B)$ and $s_{n}:\mbf Z_{n}(A/B)\to \mbf Z_{n+1}(A/B)$ by setting $r_{n}(a_{0}\otimes\cdots\otimes a_{n+1})=\mbf 1\otimes a_{0}\otimes\cdots \otimes a_{n+1}$ and  $s_{n}(a_{0}\otimes\cdots\otimes a_{n})=a_{0}\otimes\mbf 1\otimes a_{1}\otimes\cdots \otimes a_{n}$ for every $n\geq -1$. This finishes the proof.
\end{proof}

Following \cite{Hoch}, we can view the Hochschild homology of a ring extension as a sequence of derived functors. The main difficulty is that, unless $A$ is flat in the category of $B$-bimodules, the resolution \eqref{bar-resolution} is not flat in the category of $A$-bimodules. In \cite{Hoch} Hochschild solves the problem by introducing the notion of relative projective resolution. Given an extension of unital rings $R/S$ and a left $R$-module $M$, we say that $M$ is \emph{$(R/S)$-projective} if it is a direct $R$-module summand of $R\otimes_{S}K$ for some $S$-module $K$. If $M$ is a left $R$-module together with  a $(R/S)$-projective resolution $0\to P_{\bullet}\to P_{\bullet}\to M\to 0$, and $N$ is a right $R$-module, one can prove, using standard arguments in homological algebra, that the homology groups $\op{Tor}_{\bullet}^{R/S}(N,M):=H_{\bullet}(N\otimes_{R}P_{\bullet})$ does not depend on $P_{\bullet}$. Now it suffices to observe that the bar complex of a ring extension $A/B$ is projective as left module over the ring extension $A\otimes A^{o}/B\otimes A^{o}$. This implies the following:

\begin{thm}[\cite{Hoch}]
 Let $A/B$ be an extension of unital rings together with a $A$-bimodule $M$. If we write $R=A\otimes A^{o}$ and $S=B\otimes A^{o}$ then we have
\begin{equation*}
H_{\bullet}(A/B:M)=\op{Tor}^{R/S}_{\bullet}(M,A)
\end{equation*}
\end{thm}

Let $A/B$ be a tracial extension, and let us consider the $A$-bimodule 
\begin{equation}\label{}
A\hat\otimes_{B}A=(A\otimes_{B}A)\otimes_{A*_{B}A}(A\st_{B}A)
\end{equation}
The \emph{Hochschild $L^{2}$-complex} of $A/B$ is defined as the $A\hat\otimes_{B}A$-valued Hochschild complex of $A/B$, $C^{\2}_{\bullet}(A/B):=C_{\bullet}(A/B:A\hat\otimes_{B}A)$ which is, by construction, a presimplicial left module over the von Neumann fiber square $A\st_{B}A$. The generalized von Neumann dimensions of the corresponding homology groups
\begin{equation}\label{eq:Betti-numbers}
\beta_{\bullet}^{(2)}(A/B)=\dim_{A\st_{B}A}H^{(2)}_{\bullet}(A/B)
\end{equation}
will be called the \emph{$L^{2}$-Betti numbers} of $A/B$. 

We can easily see that $L^{2}$-homology and Betti numbers of tracial algebras in the sense of \cite{ConSh} are a particular case of the above invariants. Indeed, every tracial algebra can be viewed as a tracial extension of the field of complex numbers $A/\C$. The corresponding algebraic fiber square is naturally isomorphic to $A^{o}\otimes A$, while the von Neumann fiber square $A\st_{\C}A$ corresponds to the enveloping von Neumann algebra $N\botimes N^{o}$ for $N=W^{*}(A)$. Hence $H^{\2}_{\bullet}(A/\C)=H_{\bullet}(A/\C:N\botimes N^{o})$, which is exactly the definition of the $L^{2}$-homology groups of $A$ given in \cite{ConSh}. 

\subsection{The compression formula}
Let $A$ be a finite factor together with a nonzero projection $p\in B'\cap A$. Then $B_{p}:=pBp$ is a von Neumann algebra with unit $p$ and $A_{p}:=pAp$, endowed with the normal trace $\tr_{A_{p}}(pap)=\tr_{A}(a)/\tr_{A}(p)$, is a von Neumann tracial extension of $B_{p}$.

\begin{thm}\label{thm:compression-formula}
Let $A/B$ and $p\in A\cap B'$ as above. Then we have 
 \begin{equation*}
\beta^{(2)}_{\bullet}(A_{p}/B_{p})=\frac{1}{\op{tr}_{B}\left[\mbf E( p )^{2}\right]}\;\beta^{(2)}_{\bullet}(A/B)
\end{equation*}
where $\mbf E:A\to B$ is the conditional expectation of  $A/B$.
\end{thm}
\begin{proof}    
The first part of the proof is essentially the same as that of \cite[Theorem 4.2]{ConSh}. In our case we need to replace usual projective resolutions by relative projective resolutions. More precisely, let us consider the ring extension $R/S$, where $R=A\otimes A^{o}$ and $S=B\otimes A^{o}$, together with the projection $q=p\otimes p^{o}\in R$. The functor $M\to M_{p}:=qM$ from the category of $R$-modules to the category of $R_{q}$-modules (where $R_{q}=qRq$) is exact and maps $(R/S)$-projective modules to $(R_{q}/S_{q})$-projective modules. It follows that $q\mbf K_{\bullet}(A/B)=p\mbf K_{\bullet}(A/B)p$ is a $(R_{q}/S_{q})$-projective resolution of $qA=A_{p}$ and hence, for every $A$-bimodule $M$ we have
\begin{equation}\label{}
H_{\bullet}(A_{p}/B_{p}:M_{p})=H_{\bullet}(A/B:M).
\end{equation}

Now recall that $A\otimes_{B}A$ carries commuting $A$-bimodule and left $A\ast_{B}A$-module structures, which fit together into a $(A\ast_{B}A)\otimes A\otimes A^{o}$-module structure. The natural map $A_{p}\otimes A_{p}\to A\otimes A$ induces an isometric embedding of hermitian spaces $\theta:A_{p}\otimes_{B_{p}}A_{p}\to A\otimes_{B}A$ and can one easily verify that $A_{p}*_{B_{p}}A_{p}=\theta^{*}(A*_{B}A)\theta$. The corresponding orthogonal projection $\theta\theta^{*}$ is given on elementary tensors by the formula $\theta\theta^{*}(a\otimes b)= pap\otimes pap$. In other words, if $p* p$ denotes the projection of $A\ast_{B}A$ defined  by $(p\ast p)(a\otimes b)=ap\otimes pb$, then we have $\theta\theta^{*}=(p\ast p)\otimes p\otimes p^{o}$. Hence, the algebraic fiber square $A_{p}\ast_{B_{p}}A_{p}$ is spatially isomorphic to the compression of $A\ast_{B}A$ with respect to $p\ast p$, and we get an isomorphism of $A_{p}\ast_{B_{p}}A_{p}$-modules  between $A_{p}\hat\otimes_{B_{p}}A_{p}$ and $(p\ast p) (A\hat\otimes_{B}A)_{p}$. Therefore 
\begin{align*}\label{}
H_{\bullet}^{(2)}(A_{p}/B_{p})
	&=H_{\bullet}(A_{p}/B_{p}:(p*p)(A\hat\otimes_{B}A)_{p}) \\
	&\simeq (p* p)H_{\bullet}(A_{p}/B_{p}:(A\hat\otimes_{B}A)_{p})\\
	&\simeq (p*p)H_{\bullet}^{(2)}(A/B).
\end{align*}
 The result now follows from the well known formula $\dim_{W_{p}}pM=\tr_{W}(p)\dim_{W}M$ together with the identity $\op{tr}_{A*_{B}A}(p* p)=\op{tr}_{B}(\mbf E(p)^{2})$, which can easily be deduced from \eqref{eq:inner-tensor}.
\end{proof}

Note that this formula generalizes the compression formulas of Connes-Shlyakhtenko \cite[Theorem 2.4]{ConSh} and Gaboriau \cite[Th\'eor\`eme 5.3]{GabL2} (see also \cite[Theorem 1.1]{Sau} and \cite{MR2501006}). Indeed, assume that $p\in A$ is a projection with trace $\lambda$. If it belongs to the center of $B$, then $\mbf E(p)=p$ and we get a Gaboriau type compression formula
\begin{equation*}
\beta^{(2)}_{\bullet}\left(A_{p}/B_{p}\right)=\frac 1{\lambda}\;\beta^{(2)}_{\bullet}(A/B).
\end{equation*}
In the other side, if $B$ is reduced to the complex numbers, then any projection $p\in A$ verifies the hypothesis of the theorem and the conditional expectation is none other than the trace itself. Hence we obtain the Connes-Shlyakhtenko compression formula
\begin{equation*}
\beta^{(2)}_{\bullet}\left(A_{p}\right)=\frac{1}{\lambda^{2}}\;\beta^{(2)}_{\bullet}(A).
\end{equation*}

\begin{rem}
Theorem \ref{thm:compression-formula} can be easily generalized to the case of a general tracial extension $A/B$ (i.e. when only $B$ is a von Neumann algebra) by assuming $p\in Z(B)=B'\cap B$, and the formula probably holds for any $p\in B'\cap A$.
\end{rem}

\begin{cor}
 If $A$ has a 
\end{cor}

\subsection{Directed sums}
$L^{2}$-Betti numbers of tracial $*$-algebras in the sense of Connes-Shlyakhtenko  have a strange quadratic behavior with respect to direct sums, which is explained in \cite{ConSh} by the fact that they use bimodules instead of modules. As we shall see here, the reason is mainly that they do not take into account the essential role of the underlying field. 

\begin{thm}
Let $A_{n}/B_{n}$ be a finite set of tracial extensions and $(\alpha_{n})$ a partition of unity. If $A=\sum_{n}\alpha_{n}A$ and $B=\sum_{n}\alpha_{n}B_{n}$ then we have and isomorphism of presimplicial $A*_{B}A$-modules $C_{\bullet}^{\2}(A/B)\simeq\bigoplus_{n}C^{\2}_{\bullet}(A_{n}/B_{n})$. In particular 
\begin{equation}\label{eq:iso-homology-directed-sums}
H^{(2)}_{\bullet}(A/B)\simeq \bigoplus_{n}H^{(2)}(A_{n}/B_{n})
\end{equation}
and 
\begin{equation}\label{}
\beta_{\bullet}^{(2)}(A/B)=\sum_{n}\alpha_{n}\beta_{\bullet}^{(2)}(A_{n}/B_{n})
\end{equation}
\end{thm}
\begin{proof}
We have an obvious identification of $A$-bimodules between $A^{\otimes_{B}k}$ and $\bigoplus_{n}A_{n}^{\otimes_{B_{n}}k}$, which yields to an isomorphism of presimplicial $A$-bimodules $\mbf K_{\bullet}(A/B)\simeq \bigoplus_{n}\mbf K_{\bullet}(A_{n}/B_{n})$. Since $A*_{B}A=\bigoplus_{n}(A_{n}*_{B_{n}}A_{n})$ by proposition \ref{prop:fiber-product-directed-sums}, it follows that $A\hat\otimes_{B}A=\bigoplus_{n}A_{n}\hat\otimes_{B_{n}}A_{n}$ and hence
\begin{equation}\label{}
C_{\bullet}^{(2)}(A/B)\simeq\bigoplus_{n}C^{(2)}_{\bullet}(A_{n}/B_{n})
\end{equation}
This finishes the proof. 
\end{proof}

In the case of central extensions of a fixed abelian von Neumann algebra, we recover the same unnatural quadratic additivity that in \cite{ConSh}.

\begin{prop}
 Let $B$ be a commutative von Neumann algebra and $A_{n}$ a finite set of central tracial extensions of $B$. If $(\alpha_{n})$ is a partition of unity and $A=\sum_{n}\alpha_{n}A_{n}$ then
\begin{equation}\label{}
H^{(2)}_{\bullet}(A/B)\simeq \bigoplus_{n}H^{(2)}(A_{n}/B)
\end{equation}
and
\begin{equation}\label{}
\beta_{\bullet}\left(A/B\right)=\sum_{n}\alpha_{n}^{2}\beta_{\bullet}^{(2)}(A_{n}/B)
\end{equation}
\end{prop}
\begin{proof}
Replacing $A$-projective resolutions by $(A/B)$-projective resolutions and using Proposition \ref{prop:fiber-product-central-directed-sums}, the proof follows verbatim that of \cite[Proposition 2.5]{ConSh}. 
\end{proof}

Observe that, if we set $\beta^{(2)}_{\bullet}(A)$ to be $\beta^{(2)}(A/Z(A))$ instead of $\beta^{(2)}_{\bullet}(A/\C)$, then we have
\begin{equation}\label{}
\beta_{\bullet}^{(2)}\left(\sum\alpha_{n}A_{n}\right)=\sum_{n}\alpha_{n}\beta^{(2)}_{\bullet}(A_{n})
\end{equation}

\section{$L^{2}$-invariants of standard groupoids}\label{section:groupoids}

\subsection{Standard spaces}\label{sec:standard-bundles}
Let us consider the category of standard Borel spaces together with countable-to-one Borel maps. 

\begin{thm}[Lusin's selection theorem]\label{thm:Lusin-selection}
 For every countable-to-one Borel map $\pi:X\to Y$ there exists a countable Borel partition $X=\coprod_{k}X_{k}$ such that $\pi|_{X_{k}}$ is injective.
\end{thm}

If $\pi:X\to Y$ is a countable-to-one Borel map, and $\mu_{X}$ and $\mu_{Y}$ are two Borel measures on $X$ and $Y$ respectively, then we say that $\pi$ is \emph{measure preserving} if for every Borel set $A\subset X$ we have
\begin{equation}\label{}
\mu_{X}(A)=\int_{Y}\abs{\pi^{-1}(x)} d\mu_{Y}(x)
\end{equation}
where $\abs{\cdot}$ denotes the counting measure. The measure $\mu_{X}$ is then completely determined by $\mu_{Y}$, and it follows from Lusin's selection theorem that, if $\pi$ is surjective, the measure $\mu_{Y}$ is determined by $\mu_{X}$. Two measure preserving countable-to-one maps are said to be equivalent if they coincide almost everywhere. We shall note $\op{\bf Stan}$ the category of standard Borel spaces endowed with $\sigma$-finite Borel measures, and where morphisms are equivalence classes of measure preserving standard maps. Recall that two nonatomic standard spaces are isomorphic if and only if they have the same mass, i.e. if the measure of their whole spaces are equal. In other words, every standard space $X$ is isomorphic to the interval $[0,\lambda[$ endowed with the Lebesgue measure, for some $\lambda\in ]0,+\infty]$. The objects of $\op{\bf Stan}$ shall be called \emph{standard spaces} and its morphisms \emph{standard maps}. For every $p\in [1,\infty]$, we shall write $L^{p}X$ for $L^{p}(X,\mu_{X})$ and we shall note $L^{0}X$ the $*$-algebra of equivalence classes of possibly unbounded complex valued measurable fonctions on $X$. It contains $L^{\infty}X$ as $*$-subalgebra and $L^{p}X$ as $L^{\infty}X$-submodule.

The pair $(U,\pi)$ formed by a standard space $U$ together with a standard map $\pi:U\to X$ will be called a \emph{standard bundle over $X$}. The bundle map $\pi$ induces an homomorphism of algebras $\pi^{*}:L^{\infty}X\to  L^{\infty}U$ and hence a structure of $L^{\infty}X$-module on $L^{p}U$ for every $p$. A standard bundle $(U,\pi)$ is said to be \emph{bounded} if the function 
\begin{equation*}
\abs{\pi^{-1}}:x\in X\mapsto \abs{\pi^{-1}(x)}
\end{equation*}
belongs to $L^{\infty}X$. Let $\C[U, \pi]$ be the vector space of all the functions $f\in L^{\infty}U$ with bounded support with respect to $\pi$, i.e. such that the restriction of $\pi$ to the support of $f$ defines a bounded bundle. It is obviously a $L^{\infty}X$-submodule of $L^{\infty}U$ and we have a natural $L^{\infty}X$-valued inner product given by
\begin{equation}\label{}
(f^{*}g)(x)=\sum_{\pi(u)=x}\overline {f(u)}g(u)
\end{equation}
and a $L^{\infty}X$-linear map $\C[\pi]:\C[U,\pi]\to L^{\infty}X$ given by 
\begin{equation}\label{}
\C[\pi](f)(x)=\sum_{\pi(u)=x}f(u)
\end{equation}

More generally we define a {\em standard multibundle} $(U,\mc B_{U})$ as a standard space $U$ endowed with a finite set $\mc B_{U}=\{\pi_{1},\dots,\pi_{n}\}$ of standard maps $\pi_{i}:U\to X$. Each $\pi\in \mc B_{U}$ induces $L^{\infty}X$-module structure on $L^{p}U$, and the corresponding multi-module\footnote{A multi-module over a commutative algebra $B$ is by definition a $B\otimes\overset{n}\dots\otimes B$-module.} is denoted by $L^{p}[U]$. It contains $\C[U]=\bigcap_{\pi\in \mc B_{U}} \C[U,\pi]$ as a sub-multi-module. A morphism of multi-bundles is a map $\phi:U\to V$ such that $\sigma\phi\in \mc B_{U}$ for every $\sigma\in \mc B_{V}$. The category of multi-bundles over $X$ together with their morphisms shall be noted $\op{\bf MBun}_{X}$. We shall note $\op{\bf MMod}_{L^{\infty}X}$ the category of multi-modules over $L^{\infty}X$.

\begin{lem}\label{lem:functor-multibundle-multimodule}
 Let $\phi:U\to V$ be a morphism of multi-bundles. Then every $f\in \C[U]$ has bounded support with respect to $\phi$ and $\C[\phi](f)\in \C[V]$. In particular, $\C[-]$ is a covariant functor from $\op{\bf MBun}_{X}$ to $\op{\bf MMod}_{L^{\infty}X}$
\end{lem}
\begin{proof}
 Let us note $S_{x}^{\pi}f=\{u\in U:\pi(u)=x,f(u)\neq 0\}$. Recall that $f\in \C[U]$ if and only if $\op{esssup}_{x}\# S_{x}^{\pi}f<+\infty$ for every $\pi\in \mc B_{U}$. Note that for every $\sigma\in \mc B_{V}$ and every $v\in V$ one has $S_{v}^{\phi}f\subset S_{\sigma(v)}^{\sigma\phi}f$ and hence
\begin{equation}\label{}
\op{esssup}_{v}\# S_{v}^{\phi}f\leq \op{esssup}_{x}\# S_{x}^{\sigma\phi}f<+\infty
\end{equation}
This shows that $f$ has bounded support with respect to $\phi$, such that we have a well defined function $\C[\phi](f)\in L^{\infty}V$. It remains to prove that $g=\C[\phi](f)$ has bounded support with respect to every $\sigma\in \mc B_{V}$. For this let us note that $S_{x}^{\sigma}g\subset \phi(S_{x}^{\sigma\phi}f)$. Hence 
\begin{equation}\label{}
\op{esssup}_{x}\# S_{v}^{\sigma}g\leq \op{esssup}_{x}\# S_{x}^{\sigma\phi}f<+\infty
\end{equation}
which finishes the proof. 
\end{proof}

\begin{prop}\label{prop:multi-bundle-decomposition}
 For every standard multi-bundle $U$ over $X$, there exists a sequence of multibundles $Y_{k}$ such that $Y_{k}\subset X$, $\mc B_{Y_{k}}\subset \Aut(X)$ and
 \begin{equation}\label{}
\C[U]\simeq \bigoplus_{k}\C[Y_{k}]
\end{equation} 
In particular, $\C[U]$ is projective with respect to all its $L^{\infty}X$-module structures.
\end{prop}
\begin{proof}
Let $\mc B_{U}=\{\pi_{1},\dots,\pi_{n}\}$. By Lusin's selection theorem we can write $U=\coprod_{k}U_{k}$ such that every $\pi_{i}$ is injective over each $U_{k}$, which implies that $\C[U]=\bigoplus_{k}\C[U_{k}]$. Let $Y_{k}=\pi_{1}(U_{k})\subset X$. The injective standard maps $\pi_{i}\pi_{1}^{-1}:Y_{k}\to X$ can be completed into automorphisms $\phi_{k,i}:X\to X$. A simple calculation shows that $\C[\pi_{1}]$ induces an isomorphism between $\C[U_{k}]$ and $\C[Y_{k}]$, where $\mc B_{Y_{k}}=\{\phi_{k,1},\dots,\phi_{k,n}\}$. 
\end{proof}

Let $U$ and $V$ be two standard multi-bundles over $X$ and let $\pi\in \mc B_{U}$ and $\sigma\in \mc B_{V}$. The {\em fiber product} of $U$ and $V$ with respect to $\pi$ and $\sigma$ is the standard space 
$$
U{_{\pi}\star_{\sigma}} V=\set{(u,v)\in U\times V| \pi(u)=\sigma(v)}
$$ 
endowed with the sequence of standard maps $\mc B_{U{_{\pi}\star_{\sigma}} V}=\{\bar\tau|\tau\in \mc B_{U}\cup\mc B_{V}\}$ defined by
\begin{equation*}
\bar\tau(u,v)=
\begin{cases}
 \tau(u) &\text{if $\tau\in \mc B_{U}$} \\
 \tau(v) &\text{if $\tau\in \mc B_{V}$}
\end{cases}
\end{equation*}
Note that, since $\bar\pi=\bar\sigma$, we have $\abs{\mc B_{U{_{\pi}\star_{\sigma}} V}}\leq \abs{\mc B_{U}}+\abs{\mc B_{V}}-1$. We define the {\em star product}
\begin{equation}\label{eq:star-product}
L^{0}U\times L^{0}V\to L^{0}(U\star V)\quad (f,g)\mapsto f\star g
\end{equation}
by the formula $(f\star g)(u,v)=f(u)g(v)$. If we note $\coprod$ the disjoint union, then we have 
\begin{equation}\label{eq:star-product-disjoint-union}
\left(\coprod_{i}U_{i}\right){_{\coprod\pi_{i}}\star_{\coprod\sigma_{i}}}\left(\coprod_{j} V_{j}\right)=\coprod_{i,j} (U_{i}){_{\pi_{i}}\star_{\sigma_{i}}}(V_{j})
\end{equation}

\begin{prop}\label{prop:star-vs-tensor}
Let $U$ and $V$ be two standard multi-bundles over $X$. For every $\pi\in\mc B_{U}$ and $\sigma\in \mc B_{V}$, the star product induces an isomorphism of multimodules  
 \begin{equation*}
\C [U]{_{\pi}\otimes_{\sigma}} \C [V]\to \C[U{_{\pi}\star_{\sigma}} V]
\end{equation*}
where ${_{\pi}\otimes_{\sigma}}$ denotes the tensor product with respect to the $L^{\infty}X$-module structures induced by $\pi$ and $\sigma$. Further, if $\phi:U\to U'$ and $\psi:V\to V'$ are morphisms of multibundles then $\phi\star\psi:(u,v)\mapsto (\phi(u),\psi(u))$ maps $U{_{\pi}\star_{\sigma}} V$ into $U'{_{\pi\phi}\star_{\sigma\phi}} V'$ and we have $\C[\phi \star\psi]=\C[\phi]\otimes\C[\psi]$
\end{prop}
\begin{proof}
The only non trivial part of the statement is the bijectivity of the map induced by the star product. The injectivity follows from the fact that it preserves the $L^{\infty}X $-valued inner products, which are known to be non degenerated. Indeed, we have
\begin{align*}
[(f\otimes a)^{*}(g\otimes b)](x) &=(f^{*}g)(x)\cdot (a^{*}b)(x) \\ 
&=\left(\sum_{\pi(u)=x}\overline{f(u)}g(u)\right)\left(\sum_{\sigma(v)=x}\overline{a(v)}b(v)\right) \\
&=\sum_{\hat\pi(u,v)=x}\overline{f(u)a(v)}g(u)b(v) = (f\star a)^{*}(g\star b) 
\end{align*}
It remains to prove that every function $h$ on $U{_{\pi}\star_{\sigma}} V$ with bounded support with respect to $\mc B_{U{_{\pi}\star_{\sigma}}V}$ can be written as a finite sum of functions of the form $f\star g$, where $f$ and $g$ have bounded support with respect to $\mc B_{U}$ and $\mc B_{V}$ respectively. Let us note $W\subset U{_{\pi}\star_{\sigma}} V$ the support of $h$. Since $W$ is bounded with respect to all the bundle maps of $U{_{\pi}\star_{\sigma}}V$ there exists, by Lusin's selection theorem, a finite partition $\set{W_{\tau}:\tau\in \mc B_{U{_{\pi}\star_{\sigma}}V}}$ of $W$ such that $\bar\tau$ is injective on $W_{\tau}$ for every $\tau$. The restriction of $h$ to $W_{\tau}$ takes the form $h_{\tau}\star 1$ if $\tau\in \mc B_{U}$ and the form $\mbf 1\star h_{\tau}$ if $\tau\in \mc B_{V}$. It follows that $h$ is a sum of functions of the form $f\star g$, as wanted. 
\end{proof}

Let $U$ be a standard multi-bundle. For any pair of bundle maps $\pi,\sigma\in \mc B_{U}$ we note $U^{\pi\sigma}$ the biggest sub-multi-bundle of $U$ where $\pi$ and $\sigma$ coincide. More precisely we set $U^{\pi\sigma}=\{u\in \mc U|\pi(u)=\sigma(u)\}\subset U$ and $\mc B_{U^{\pi\sigma}}=\{\tau|_{U^{\pi\sigma}}:\tau\in \mc B_{U}\}$.

\begin{prop}\label{prop:multi-modules-fulness}
Let $U$ and $\pi,\sigma\in \mc B_{U}$ be as above. Then $\C[U^{\pi\sigma}]$ coincides with the space of invariant elements of $\C[U]$ with respect to the $L^{\infty}X$-bimodule structure induced by $\pi$ and $\sigma$. Furthermore, the natural map from invariants to coinvariants $\Psi:\C[U]^{L^{\infty}X}\to \C[U]_{L^{\infty}X}$ is an isomorphism of multi-modules.
\end{prop}
\begin{proof}
 By proposition \ref{prop:multi-bundle-decomposition} it suffices to prove the result for $U\subset X$ and $\mc B_{U}\subset \Aut(X)$. In that case $f\in \C[U]^{\bds L^{\infty}}$ if and only if $(a\pi-a\sigma)f=0$ for every $a\in L^{\infty}X$, which is equivalent to say that the support of $f$ is contained in $U^{\pi\sigma}$. This proves the first assertion. For the second one, we shall consider $a\in L^{\infty}X$ such that $b=a\pi-a\sigma$ is invertible in $L^{\infty}X$. Hence for every $f\in L^{\infty}Y$ we have $f=f|_{U^{\pi\sigma}}+[a,h]$ where $h=b^{-1}f|_{U-U^{\pi\sigma}}$. 
\end{proof}

\subsection{Standard groupoids}
Let $X$ be a standard probability space. A \emph{standard groupoid} over $X$, also known as a discrete measured groupoid (cf. \cite{Sau}), is given by a standard space $\mc G$ endowed with two standard maps $s,t:\mc G\to X$, called the source and the target respectively, an involutive isomorphism of standard spaces $\alpha\in\mc G\mapsto \alpha^{-1}\in \mc G$ called the inversion, and a standard map $\mc G{_{s}\star_{t}}\mc G\to\mc G$, called the {\em composition} and noted $(\alpha,\beta)\mapsto\alpha\beta$, such that $s(\alpha^{-1})=t(\alpha)$, $(\alpha\beta)\gamma=\alpha(\beta\gamma)$ and $\alpha(\beta\beta^{-1})=\alpha=(\gamma\gamma^{-1})\alpha$, whenever it makes sense. Given two standard groupoids $\mc G$ and $\mc S$ over $X$, a standard map $\phi:\mc G\to \mc S$ is called a morphism if it preserves the source, target, composition and inverse maps. The kernel of $\phi$, defined by $\ker\phi=\{\alpha\in \mc G|\phi(\alpha)=1\}$, and its image $\im\phi$, are subgroupoids of $\mc G$ and $\mc S$ respectively. In that case we shall say that $\mc G$ is the semi-direct product of $\im\phi$ and $\ker \phi$ and we shall write $\mc G=\im\phi\ltimes \ker\phi$. Every standard groupoid has a natural decomposition as a semi-direct product. Indeed, the map $\phi=(t,s):\mc G\to X\times X$ is a morphism from $\mc G$ to the non standard groupoid $X\times X$. Its image is a standard groupoid called the orbit equivalence relation of $\mc G$. The kernel of $\phi$ is called the {\em isotropy groupoid} of $\mc G$.

\begin{ex}
 The trivial standard groupoid (over $X$), denoted by $\mbf 1_{X}$, is defined by $X$ itself endowed with the identity maps.
\end{ex}

\begin{ex}
 A standard groupoid is called a {\em standard equivalence relation} if its isotropy groupoid is trivial.
\end{ex}

\begin{ex}
 A standard groupoid is called a {\em measurable field of countable groups} if its orbit equivalence relation is trivial or, in other words, if its target and source maps coincide. 
\end{ex}

Let $\mc G$ be a standard groupoid. A {\em $\mc G$-space} is a standard multi-bundle $U$ endowed with a fixed bundle map $\pi\in \mc B_{U}$ and a morphism of multi-bundles, called the {\em action},
\begin{equation*}
 \mc G{_{s}\star_{\pi}} U\to U\quad,\quad (\alpha,u)\mapsto \alpha u
\end{equation*}
 such that $\alpha(\beta u)=(\alpha\beta) u$ whenever this make sense. For instance $\mc G$ is a $\mc G$-space where the action is the composition, and $X$ is a $\mc G$ space with respect to the action $\alpha x=t(\alpha)$ for $x=s(\alpha)$. By proposition \ref{prop:star-vs-tensor} we have a natural map $\C[\mc G]{_{s}\otimes_{\pi}}\C[U]\to \C[U]$ noted $f\otimes g\mapsto f\cdot g$ and called the {\em dot product}. If we set $U=\mc G$, then the dot product induces a structure of unital algebra on $\C[\mc G]$ and, more generally, a structure of $\C\mc G$-module on $\C[U]$. The involution $f^{*}(\alpha)=\overline{f(\alpha^{-1})}$ turn $\C[\mc G]$ into a $*$-algebra called the {\em convolution algebra} of $\mc G$ and noted $\C\mc G$. If $\Delta$ denote the Borel subset of $\mc G$ formed by the identity elements, i.e. by elements of the form $\alpha\alpha^{-1}$, then $\Delta$ can be identified to $X$ via the source or the target maps. The dot product induces the usual point-wise product over $L^{\infty}\Delta\subset \C\mc G$, and such a pair forms a tracial extension with conditional expectation $\mbf E:\C\mc G\to L^{\infty}\Delta$ given by $\mbf E(f)=f|_{\Delta}$ and trace $\tr_{\C\mc G}(f)=\int_{\Delta}\mbf E(f)$. We shall write $\bds L^{\infty}$ for $L^{\infty}\Delta$ and $N\mc G$ for the von Neumann algebra $W^{*}(\C\mc G)$. 

\medskip
In order to define $L^{2}$-invariants of standard groupoids and understand the relation with their algebraic partners, we must introduce a series of presimplicial multibundles:
\begin{enumerate}
\item The {\em nerve} of a standard groupoid $\mc G$ is defined by the sequence of standard multi-bundles $\mc G^{(n)}=\mc G{_{s}\star_{t}}\overset{n}\dots{_{s}\star_{t}} \mc G$ (with $\mc G^{(0)}=X$ by convention) together with the maps $\mbf n_{n,i}:\mc G^{(n)}\to \mc G^{(n-1)}$ defined by  
\begin{equation}\label{eq:nerve-face-maps}
\mbf n_{n,i}(\alpha_{1},\dots,\alpha_{n})=
\begin{cases}
 (\alpha_{2},\dots,\alpha_{n}) & i=0\\
 (\alpha_{1},\dots,\alpha_{i}\alpha_{i+1},\dots,\alpha_{n})& 1\geq i\geq n-1\\
 (\alpha_{1},\dots,\alpha_{n-1}) & i=n
\end{cases}
\end{equation}

\item The {\em bar space} of $\mc G$  is the sequence $\mbf K^{n}\mc G=\mc G^{(n+2)}$ with face maps $\mbf k'_{n,i}=\mbf n_{n+2,i+1}$;

\item The {\em cyclic space} of $\mc G$ is defined by the sequence $\mbf H^{n}\mc G=(\mc G{_{s}\star_{t}}\overset{n+1}\dots{_{s}\star_{t}} \mc G)_{st}$ together with the maps
\begin{equation}\label{eq:cyclic-face-maps}
\mbf h'_{n,i}(\alpha_{0},\dots,\alpha_{n})=
\begin{cases}
(\alpha_{0},\dots,\alpha_{i}\alpha_{i+1},\dots,\alpha_{n}) &:i=0,\dots,n-1\\
(\alpha_{n}\alpha_{0},\alpha_{1},\dots,\alpha_{n-1}) & : i=n 
\end{cases}
\end{equation}

\item The {\em acyclic space} of $\mc G$ is given by $\mbf Z^{n}\mc G=\mbf H^{n+1}\mc G$ together with the maps $\mbf z'_{n,i}=\mbf h'_{n+1,i+1}$.

\item The {\em classifying space} of $\mc G$ is given by the sequence $\mbf E^{n}\mc G:=\mc G{_{t}\star_{t}}\overset{n+1}\dots{_{t}\star_{t}} \mc G$ with face maps 
\begin{equation}\label{}
\partial_{n,i}(\alpha_{0},\dots,\alpha_{n})=(\alpha_{0},\dots,\what{\alpha_{i}},\dots,\alpha_{n})
\end{equation}
\end{enumerate}

\begin{prop}\label{prop:C-algebra-functor}
 Every morphism of standard groupoids $\phi:\mc G\to \mc S$ induces a morphism of their full, bar, cyclic, acyclic and classifying spaces and complexes. Moreover  $\C\phi:\C\mc G\to \C\mc S$ is a morphism of tracial extensions that induces the identity on $\bds L^{\infty}$.
\end{prop}
\begin{proof}
All the assertions can be easily deduced from proposition \ref{prop:star-vs-tensor}.
\end{proof}
 
\begin{prop}\label{prop:iso-bar-cyclic-complexes-groupoids}
The dot product induces a $\C\mc G$-bimodule structure on $C(\mc G)$ which is preserved by the face maps of the geometric bar complex $C_{\bullet}(\mbf K\mc G)$. Furthermore:  
\begin{enumerate}
\item $C_{\bullet}(\mbf K\mc G)$ is isomorphic to the Hochschild bar complex of the extension $\C\mc G/\bds L^{\infty}$ in the category of presimplicial $\C\mc G$-bimodules;
\item The cyclic complex of $\mc G$ is isomorphic to the Hochschild complex of $\C\mc G/\bds L^{\infty}$ in the category of presimplicial $\bds L^{\infty}$-multi-modules;
\end{enumerate}
\end{prop}
\begin{proof}
By proposition \ref{prop:star-vs-tensor}, the star product induces natural isomorphisms
\begin{equation}\label{}
C_{n}(\mbf K\mc G)=\C[\mc G{_{s}\star_{t}}\overset{n+2}\dots{_{s}\star_{t}}\mc G]\simeq \C\mc G{_{s}\otimes_{t}}\overset{n+2}\dots{_{s}\otimes_{t}}\C\mc G =\mbf K_{n}(\C\mc G/\bds L^{\infty})
\end{equation}
while by proposition \ref{prop:multi-modules-fulness} we have
\begin{multline}\label{}
C_{n}(\mbf H\mc G)
	=\C[(\mc G{_{s}\star_{t}}\overset{n+1}\dots{_{s}\star_{t}}\mc G)^{st}]\\
	=\C[\mc G{_{s}\star_{t}}\overset{n+1}\dots{_{s}\star_{t}}\mc G]^{\bds L^{\infty}}
	\simeq \C[\mc G{_{s}\star_{t}}\overset{n+1}\dots{_{s}\star_{t}}\mc G]_{\bds L^{\infty}} \\
	\simeq (\C\mc G{_{s}\otimes_{t}}\overset{n+1}\dots{_{s}\otimes_{t}}\C\mc G)_{\bds L^{\infty}}
	= C_{n}(\C\mc G/\bds L^{\infty})
\end{multline}
Now it is obvious, in view of formulas \eqref{eq:bar-face-maps}, \eqref{eq:Hochschild-face-maps} and \eqref{eq:cyclic-face-maps}, that the above isomorphisms preserve the face maps. 
\end{proof}

The structure of $\C\mc G$-module induced by the dot product on $C_{0}(\mc G)=L^{\infty}X$ has been used by Roman Sauer in \cite{Sau} to define the $L^{2}$-homology groups of $\mc G$ as the sequence of $N\mc G$-modules $H^{\2}_{\bullet}(\mc G)=\op{Tor}_{\bullet}^{\C\mc G}(L^{\infty}X,N\mc G)$. The corresponding von Neumann dimensions 
\begin{equation}\label{}
\beta_{\bullet}^{(2)}(\mc G):=\dim_{N\mc G}\op{Tor}_{\bullet}^{\C\mc G}(L^{\infty}X,N\mc G)
\end{equation}
are called the $L^{2}$-Betti numbers of $\mc G$. He proved that this definition yields to $L^{2}$-Betti numbers of standard equivalence relations, under the assumption that they are generated by essentially free actions of countable groups. The freeness hypothesis was dropped some years later by Neshveyev and Rustad in \cite{NeshRus}. 

\subsection{Proof of theorem \ref{thm:L2=L2-groupoids}}
Since $C_{\bullet}(\mbf E\mc G)$ is a $\C\mc G$-projective resolution of $\bds L^{\infty}$, one can define $H^{\2}_{\bullet}(\mc G)$ as the homology groups of the presimplicial $N\mc G$-module $N\mc G\otimes_{\C\mc G}C_{\bullet}(\mbf E\mc G)$. On the other hand $H^{\2}_{\bullet}(\C\mc G/\bds L^{\infty})$ are defined as the homology groups of $W^{*}(A)\otimes_{A}\mbf E_{\bullet}(\C\mc G/\bds L^{\infty})$ where $A$ is the fiber square $\C\mc G*_{\bds L^{\infty}}\C\mc G$. Hence, we need to compare $C_{\bullet}(\mbf E\mc G)$ to $\mbf E_{\bullet}(\C\mc G/\bds L^{\infty})$ and $\C\mc G$ to $\C\mc G*_{\bds L^{\infty}}\C\mc G$. 

In order to do so, we define the {\em enveloping  groupoid} of $\mc G$ as the standard space $\mc G^{e}=\{(\alpha,\beta)\in \mc G^{(2)}|s(\alpha)=t(\beta),t(\alpha)=s(\beta)\}$ endowed with the source and target maps $s(\alpha,\beta)=s(\alpha)=s(\beta)$ and $t(\alpha,\alpha')=t(\alpha)=t(\beta)$, while the inversion and composition are given by
\begin{equation}\label{}
(\alpha,\beta)^{-1}=(\alpha^{-1},\beta^{-1})\quad,\quad (\alpha,\beta)(\alpha',\beta')=(\alpha'\alpha,\beta\beta')
\end{equation}
Observe that $\mc G$ can be viewed as a sous-groupoid of $\mc G$ via the morphism $\alpha\mapsto (\alpha^{-1},\alpha)$. For instance, if $\mc G$ is a standard equivalence relation, then the inclusion $\mc G\subset \mc G^{e}$ is actually an isomorphism.

\begin{thm}
Let $\mc G$ a standard groupoid, then the map $T\mapsto T(\mbf 1\otimes \mbf 1)$ defines an isomorphism between $\C\mc G*_{\bds L^{\infty}}\C\mc G$ and $\C\mc G^{e}$. In particular, the von Neumann algebra $\C\mc G\st_{\bds L^{\infty}}\C\mc G$ coincides with the commutant of $\C\mc G\otimes \C\mc G^{o}$ in $L^{2}\mc G^{(2)}$. 
\end{thm}
\begin{proof}
If $\mc R$ and $\mc I$ denote respectively the orbit equivalence relation and the isotropy subgroupoid of $\mc G$, then $\mc G=\mc R{_{s}\star_{t}} \mc I$ and $\mc G^{e}=\mc R^{e}{_{s}\star_{t}} \mc I^{e}$. In particular we have $C_{1}(\mc G^{e})=C_{1}(\mc R^{e})\otimes_{\bds L^{\infty}} C_{1}(\mc I^{e})$ and $C_{2}(\mc G)=C_{2}(\mc R)\otimes_{\bds L^{\infty}} C_{2}(\mc I)$ as $\bds L^{\infty}$-bimodules. For every $u\in \mbf U_{\C\mc G/\bds L^{\infty}}$ there exists $\phi\in \Aut(X)$ such that $uau^{*}=a_{\phi}$ for every $a\in \bds L^{\infty}$. In other words, there exists $\bar u\in \mbf U_{\C\mc R/\bds L^{\infty}}$ such that $u=\bar u\otimes \gamma_{u}$ for some unitary $\gamma_{u}\in \mbf U_{\C\mc I}$. A simple computation shows then that $u*v=(\bar u*\bar v)\otimes (\gamma_{u}*\gamma_{v})$. Hence
\begin{equation}\label{}
(u*v)(\mbf 1_{\mc G}\otimes\mbf 1_{\mc G})=(\bar u*\bar v)(\mbf 1_{\mc R}\otimes\mbf 1_{\mc R})\otimes (\gamma_{u}*\gamma_{v})(\mbf 1_{\mc I}\otimes\mbf 1_{\mc I})
\end{equation}
which shows that it suffices to prove the theorem for $\mc R$ and $\mc I$. In the first case $\mc R\simeq \mc R^{e}$ via the map $\alpha\to (\alpha^{-1},\alpha)$ and the result follows from theorem \ref{thm:square-fiber-twisted-algebra}. In the case of the isotropy groupoid $\mc I$, the extension $\C\mc I/\bds L^{\infty}$ is central and the result is an immediate consequence of (iv) in Proposition \ref{prop:algebraic-fiber-product}.
\end{proof}

\begin{lem}\label{lem:iso:classifying-cyclic-isomorphism}
The presimplicial $\C\mc G^{e}$-modules 
\begin{equation}\label{}
\mbf E_{\bullet}(\C\mc G/\bds L^{\infty}),\quad C_{\bullet}(\mbf Z\mc G)\quad\text{and}\quad \C\mc G^{e}\otimes_{\C\mc G}C_{\bullet}(\mbf E\mc G)
\end{equation}
are isomorphic.
\end{lem}
\begin{proof}
The fact that $\mbf Z_{\bullet}(\C\mc G/\bds L^{\infty})$ and $C_{\bullet}(\mbf Z\mc G)$ are isomorphic as presimplicial multi-modules follows from item (ii) in proposition \ref{prop:iso-bar-cyclic-complexes-groupoids}. Hence it suffices to find an isomorphism between $\C\mc G^{e}\otimes_{\C\mc G}C_{\bullet}(\mbf E\mc G)$ and $C_{\bullet}(\mbf Z\mc G)$. One can easily check that the map $\theta_{n}:\mbf E^{n}\mc G\to \mbf Z^{n}\mc G$ given by the formula
\begin{equation}\label{eq:theta-map}
(\alpha_{0},\dots,\alpha_{n})\mapsto (\alpha_{n}^{-1},\alpha_{0},\alpha_{0}^{-1}\alpha_{1},\dots,\alpha_{n-1}^{-1}\alpha_{n})
\end{equation}
defines a morphism of presimplicial multi-bundles.  For instance $\mbf E^{0}\mc G=\mc G$, $\mbf Z^{0}\mc G=\mc G^{e}$ and $\theta_{0}:\mc G\to \mc G^{e}$ coincides with the diagonal inclusion $\alpha\mapsto (\alpha^{-1},\alpha)$. Since the action of $\mc G^{e}$ on $\mbf Z^{n}\mc G$ is given by
\begin{equation}\label{}
(\alpha,\beta)(\alpha_{0},\alpha_{1},\dots,\alpha_{n+1})=(\alpha_{0}\alpha,\beta\alpha_{1},\alpha_{2},\dots,\alpha_{n+1})
\end{equation}
and the action of $\mc G$ on $\mbf E^{n}\mc G$ is given by
\begin{equation}\label{}
\alpha(\alpha_{0},\dots,\alpha_{n})=(\alpha\alpha_{0},\dots,\alpha\alpha_{n})
\end{equation}
a straightforward computation using \eqref{eq:theta-map} shows that $\theta_{n}(\alpha\omega)=(\alpha^{-1},\alpha)\theta_{n}(\omega)$ for every $\alpha\in \mc G$ and $\omega\in \mbf E^{n}\mc G$ for which this make sense. It follows that $\C[\theta_{\bullet}]:C_{\bullet}(\mbf E\mc G)\to C_{\bullet}(\mbf Z\mc G)$ is a morphism of presimplicial $\C\mc G$-modules, and we obtain a morphism of presimplicial $\C\mc G^{e}$-modules 
\begin{equation}\label{}
\Theta_{\bullet}:\C\mc G^{e}\otimes_{\C\mc G}C_{\bullet}(\mbf E\mc G)\to \C\mc G^{e}\otimes_{\C\mc G^{e}}C_{\bullet}(\mbf Z\mc G)=C_{\bullet}(\mbf Z\mc G)
\end{equation}
by setting $\Theta_{\bullet}=\mbf 1\otimes \C[\theta_{\bullet}]$. In order to complete the proof, it suffices to show that $\Theta_{n}$ is bijective for every $n$. For this, let us observe that, up to the identifications 
\begin{equation}\label{}
\C\mc G^{e}\otimes_{\C\mc G}C_{n}(\mbf E\mc G)=\C\mc G^{e}\otimes_{\bds L^{\infty}}C_{n-1}(\mbf E\mc G)=\C[\mc G^{e}{_{s}\star_{t}}\mbf E^{n-1}\mc G]
\end{equation}
one has $\Theta_{n}=\C[\xi_{n}]$ where $\xi_{n}:\mc G^{e}{_{s}\star_{t}}\mbf E^{n-1}\mc G\to \mbf Z^{n}\mc G$ is the standard map $\xi_{n}(\gamma\star\omega)=\gamma\theta_{n}(1,\omega)$. Now a direct computation shows that the inverse of $\xi_{n}$ is given by $\omega\mapsto \phi_{n}(\omega)\star\psi_{n}(\omega)$ where
\begin{equation}\label{}
\phi_{n}(\alpha_{n+1},\alpha_{0},\dots,\alpha_{n})=(\alpha_{1}\cdots\alpha_{n+1},\alpha_{0})
\end{equation}
and 
\begin{equation}\label{}
\psi_{n}(\alpha_{n+1},\alpha_{0},\dots,\alpha_{n})=(\alpha_{1},\alpha_{1}\alpha_{2},\dots,\alpha_{n-1}\alpha_{n})
\end{equation}
This finishes the proof.
\end{proof}

Now we are able to prove theorem \ref{thm:L2=L2-groupoids}. In view of the previous lemma, we have
\begin{align*}\label{}
C^{(2)}_{\bullet}(\C\mc G/\bds L^{\infty})
	&=N\mc G^{e}\otimes_{\C\mc G^{e}}\mbf E_{\bullet}(\C\mc G/\bds L^{\infty})\\
	&=N\mc G^{e}\otimes_{\C\mc G^{e}}\C\mc G^{e}\otimes_{\C\mc G}C_{\bullet}(\mbf E\mc G)\\
	&=N\mc G^{e}\otimes_{\C\mc G}C_{\bullet}(\mbf E\mc G)
\end{align*}
Since $C_{\bullet}(\mbf E\mc G)$ is a $\C\mc G$-projective resolution of $\bds L^{\infty}$ and the functor $N\mc G^{e}\otimes_{N\mc G}-$ is flat, it follows that
\begin{align*}\label{}
H^{(2)}_{\bullet}(\C\mc G/\bds L^{\infty})
	&=\op{Tor}_{\bullet}^{\C\mc G}(\bds L^{\infty},N\mc G^{e})\\
	&=\op{Tor}_{\bullet}^{\C\mc G}(\bds L^{\infty},N\mc G)\otimes_{N\mc G}N\mc G^{e}\\
	&=H^{(2)}_{\bullet}(\mc G)\otimes_{N\mc G}N\mc G^{e}
\end{align*}
which proves the theorem.

\subsection{Proof of theorem \ref{thm:residual-L2-Betti-numbers}}  
Let $\mc R$ a standard equivalence relation and $\sigma:\mc R^{(2)}\to \mbf T$ a measurable function with values in the unit circle. For every $a,b\in C_{1}(\mc R)$ we define a bilinear product by the formula
\begin{equation}\label{}
(ab)(x,z)=\sum_{y}a(x,y)b(y,z)\sigma(x,y,z).
\end{equation}
In order to obtain an associative product $\sigma$ has to be a $2$-cocycle of $\mc R$, i.e. it verifies the relation $\sigma(x,y,z)\sigma(x,z,t)=\sigma(y,z,t)\sigma(x,y,t)$. We shall note $\C\mc R_{\sigma}$ the $*$-algebra given by $C_{1}(\mc R)$ endowed with the above twisted product and the involution $a^{*}(x,y)=\overline{a(y,x)}$. The induced product on the diagonal subalgebra $\bds L^{\infty}$ is given by $(ab)(x)=a(x)b(x)u(x)$ where $u(x)=\sigma(x,x,x)$. A cocycle is said to be a coboundary if there exists as measurable function $c:\mc R\to \mbf T$ such that $\sigma(x,y,z) = c(y,z) c(x,z)^{-1} c(x,y)$. Let us note $Z^{2}(\mc R:\mbf T)$ the abelian group of $2$-cocycles and $B^{2}(\mc R:\mbf T)$ the abelian group subgroup of coboundarie. The isomorphism class of $\C\mc R_{\sigma}$ depend only on the class of $\sigma$ in the $2$-cohomology group $H^{2}(\mc R:\mbf T)=Z^{2}(\mc R:\mbf T)/B^{2}(\mc R:\mbf T)$. Since every $2$-cocycle is cohomologous to a skew symmetric one, one may assume $u(x)=1$ and $\sigma(x,y,z)\sigma(z,y,x)=1$ (cf. \cite[page 331]{MR0578730}). We shall write $N\mc R_{\sigma}$ for $W^{*}(\C\mc R_{\sigma})$.

\begin{thm}\label{thm:square-fiber-twisted-algebra}
The fiber square of $\C\mc R_{\sigma}/\bds L^{\infty}$ does not remember the cocycle. More precisely the map $T\mapsto T(\mbf 1\otimes\mbf 1)$ defines an isomorphism between $\C\mc R_{\sigma}*_{\bds L^{\infty}}\C\mc R_{\sigma}$ and $\C\mc R^{e}$ inducing the identity on $\bds L^{\infty}$.
\end{thm}
\begin{proof}
By Proposition \ref{prop:multi-modules-fulness} one has $\C\mc R^{e}=\C[\mc R^{(2)}]^{\bds L^{\infty}}$, and it follows from (ii) in Proposition \ref{prop:algebraic-fiber-product} that $T\mapsto T(\mbf 1\otimes\mbf 1)$ is injective and take values in $\C\mc R^{e}$. Hence it suffices to prove that it is a surjective morphism of algebras. Since $\sigma$ is supposed to be skew-symmetric, the right and left actions of $\bds L^{\infty}$ on $\C[\mc R]$ do not depend on $\sigma$. This implies that $\mbf U_{\C\mc R_{\sigma}/\bds L^{\infty}}=\mbf U_{\C\mc R/\bds L^{\infty}}$ and proves that $\C\mc R_{\sigma}*_{\bds L^{\infty}}\C\mc R_{\sigma}=\C\mc R*_{\bds L^{\infty}}\C\mc R$. For every $u\in \mbf U_{\C\mc R}/\bds L^{\infty}$ there exists $\phi\in \op{Full}\mc R$ such that $u^{*}au=a\circ\phi$ for every $a\in \bds L^{\infty}$ and one can easily check that there exists an unitary function $\alpha\in \bds L^{\infty}$ such that $u=\alpha u_{\phi}$ where $u_{\phi}(x,y)=\mbf 1(x,\phi y)$. If $v$ is another unitary of $\C\mc R$ such that $(u,v)\in \mbf U_{\C\mc R/\bds L^{\infty}}$, then $v^{*}=\beta u_{\phi}$ for some unitary function $\beta\in \mbf L^{\infty}$. The groupoid isomorphism $\theta:\mc R\to \mc R^{e}$ given by $\theta(x,y) =(y,x,y)$ induces by proposition \ref{prop:C-algebra-functor} an algebra isomorphism $\C\theta:\C\mc R\to\C\mc R^{e}$ which fixes the elements of $\bds L^{\infty}$ and such that $\C\theta(u)=u^{*}\otimes u$ for every unitary $u\in \mbf U_{\C\mc R}$. Hence
\begin{multline}
(u*v)(\mbf 1\otimes \mbf 1)
	=u\otimes  v 
	=\alpha u_{\phi}\otimes u_{\phi}^{*}\beta^{*} 
	=\alpha(\C\theta(u^{*}_{\phi}))\beta^{*}= \C\theta(\alpha u^{*}_{\phi}\beta^{*})
\end{multline} 
This completes the proof.
\end{proof}

Let $A$ be a tracial von Neumann algebra together with a Cartan subalgebra $B\subset A$. Such pairs has been classified by Feldman and Moore in \cite{MR0578730}. More precisely, they prove the following structure theorem:
\begin{thm}[\cite{MR0578730}]
 For every finite factor $A$ endowed with a Cartan subalgebra, there exists an ergodic standard equivalence relation $\mc R$ (unique up to conjugacy) and a cocycle $\sigma\in Z^{2}(\mc R:\mbf T)$ (unique up to cohomology) such that the following sequence of extensions are isomorphic
\begin{equation}\label{}
A/\mc A/B\simeq N\mc R_{\sigma}/\C\mc R_{\sigma}/\bds L^{\infty}
\end{equation}
where $\mc A$ denotes the normalizing algebra of $A/B$.
\end{thm}

By theorem \ref{thm:square-fiber-twisted-algebra} there exist an algebra isomorphism $A*_{B}A= \mc A*_{B}\mc A\simeq \C\mc R$ which maps $B$ onto $\bds L^{\infty}$. Since $\mc R=\mc R^{e}$ it follows from lemma \ref{lem:iso:classifying-cyclic-isomorphism} that $\mbf E_{n}(\mc A/B)\simeq C_{n}(\mbf E\mc R)$, which is a $\C\mc R$-projective  resolution of $\bds L^{\infty}$. This proves 
\begin{equation}\label{}
H^{(2)}_{\bullet}(\mc A/B)=\op{Tor}_{\bullet}^{\C\mc R}(B,N\mc R)=H_{\bullet}^{(2)}(\mc R)
\end{equation}
which implies theorem \ref{thm:residual-L2-Betti-numbers}.

\bibliographystyle{alpha}
\bibliography{../mybib}

\end{document}